\documentclass[11pt,reqno,fleqn]{article}
\usepackage{amsfonts,amssymb,amsmath,amsthm}
\usepackage[top=20mm,bottom=20mm,left=20mm,right=20mm]{geometry}
\usepackage[T1]{fontenc}
\usepackage{times}
\usepackage{multicol}
\usepackage{amsfonts}
\usepackage{graphicx}
\usepackage{graphics}
\usepackage{enumitem}
\usepackage{relsize}
\usepackage{cite}
\usepackage{fancyhdr}
\usepackage{color}
\fancypagestyle{titlepage}
{
   \fancyhf{}
   \lhead{The final publication is available at http://www.proceedings.bas.bg/DOI/doi2021\_9\_02.html}
   \rhead{}
   \fancyfoot[C]{\thepage}
}

\newcommand\blfootnote[1]{%
  \begingroup
  \renewcommand\thefootnote{}\footnote{#1}%
  \addtocounter{footnote}{-1}%
  \endgroup
}

\theoremstyle{definition}
\newtheorem{theorem}{Theorem}[section]
\numberwithin{equation}{section}
\newtheorem{definition}[theorem]{Definition}
\newtheorem{lemma}[theorem]{Lemma} 
\usepackage{etoolbox}
\makeatletter
\patchcmd{\@maketitle}{\begin{center}}{\begin{flushleft}}{}{}
\patchcmd{\@maketitle}{\begin{tabular}[t]{c}}{\begin{tabular}[t]{@{}l}}{}{}
\patchcmd{\@maketitle}{\end{center}}{\end{flushleft}}{}{}
\makeatother
\begin{document}

\title{Basic trigonometric Korovkin approximation for fuzzy valued functions of two variables}
\author{Enes Yavuz}
\date{{\small Department of Mathematics, Manisa Celal Bayar University, Manisa, Turkey.\\ E-mail: enes.yavuz@cbu.edu.tr}}
\maketitle
\thispagestyle{titlepage}
\blfootnote{\emph{Key words and phrases:} fuzzy set theory, Korovkin approximation theorems, summability methods\\\rule{0.63cm}{0cm}\emph{\!\!Mathematics Subject Classification:} 03E72, 41A36, 40A05, 40G99}

\noindent\textbf{Abstract:}
\noindent We prove the basic trigonometric Korovkin approximation theorem for fuzzy valued functions of two variables and verify the approximation by the help of fuzzy modulus of continuity. Also, we introduce double level Fourier series of fuzzy valued functions and investigate corresponding approximation through the use of Ces\`{a}ro and Abel methods of summation of infinite series.

\section{Introduction and Preliminaries}\allowdisplaybreaks
\noindent Korovkin theorems examine approximation of continuous functions via linear operators by exhibiting finite number of test functions whose approximation ensures the approximation of all functions in the space\cite{korovkin, korovkin1}. Since their introduction, Korovkin theorems have impressed many mathematicians due to their simplicity and power when dealing with approximation of functions. Researchers have presented various test functions in different spaces of functions and provided applications concerning some known approximation processes via linear operators. For some of these results we refer to \cite{temel1,altomare4,altomare9}.

The events that humankind meets in real-world generally contain incomplete and unclear data. Researchers proposed many uncertainty theories(e.g. evidence theory, probability theory, fuzzy set theory) to handle imprecise knowledge occurring in real-world facts. Among them, fuzzy set theory was invented by Zadeh\cite{zadeh} and applied to many fields of science. In mathematics a great deal of concepts were extended to fuzzy settings. In particular, fuzzy approximation theorems were proved to approximate continuous fuzzy valued functions and corresponding rates of approximation were examined via fuzzy modulus of continuity\cite{fuzzyfejer1,ftemel4,yavuzbukres}. In this study we establish two fuzzy trigonometric Korovkin approximation theorems and apply established theorems to double level Fourier series of fuzzy valued functions by the help of some known summation procedures. Besides, uniform continuity of $2\pi$-periodic and continuous fuzzy valued functions of two variables is achieved through Lemma \ref{uniform1} and Theorem \ref{uniform2}. Now we give some preliminaries.

A \textit{fuzzy number} is a fuzzy set on the real axis, i.e. $\mu$ is normal, fuzzy convex, upper semi-continuous and $\operatorname{supp}\mu =\overline{\{t\in\mathbb{R}:\mu(t)>0\}}$ is compact \cite{zadeh}. $\alpha$-level set of $\mu$
is defined by
\begin{eqnarray*}
[\mu]_\alpha:=\left\{\begin{array}{ccc}
\{t\in\mathbb{R}:\mu(t)\geq\alpha\} & , & \qquad if \quad 0<\alpha\leq 1, \\[6 pt]
\overline{\{t\in\mathbb{R}:\mu(t)>\alpha\}} & , &
if \quad \alpha=0.\end{array} \right.
\end{eqnarray*}
$\mathbb{R}_{\mathcal{F}}$ is the set of fuzzy numbers. If $\mu,\eta\in \mathbb{R}_{\mathcal{F}}$ and $k\in\mathbb{R}$, then
\begin{eqnarray*}
[\mu+ \eta]_\alpha=[\mu]_\alpha +[\eta]_\alpha=[\mu^-_{\alpha}+\eta^-_{\alpha}, \mu^+_{\alpha}+\eta^+_{\alpha}]\quad ,\quad [k \mu]_\alpha=k[\mu]_\alpha
\end{eqnarray*}
where $[\mu]_\alpha=[\mu^-_{\alpha}, \mu^+_{\alpha}]$, for all $\alpha\in[0,1]$. Partial ordering on $\mathbb{R}_{\mathcal{F}}$ is given by:
\begin{eqnarray*}
\mu\preceq \eta \Longleftrightarrow
[\mu]_\alpha\preceq[\eta]_\alpha\Longleftrightarrow \mu^-_{\alpha}\leq
\eta^-_{\alpha}~\text{ and }~\mu^+_{\alpha}\leq \eta^+_{\alpha}~\text{ for
all }~\alpha\in[0,1].
\end{eqnarray*}
The metric $D$ on $\mathbb{R}_{\mathcal{F}}$ is defined by
\begin{eqnarray*}
 D(\mu,\eta):=
\sup_{\alpha\in[0,1]}\max\{|\mu^-_{\alpha}-\eta^-_{\alpha}|,|\mu^+_{\alpha}-
\eta^+_{\alpha}|\}.
\end{eqnarray*}
A function $f$ on $\mathbb{R}^2$ is $2\pi$-periodic if
\begin{eqnarray*}
f(u,v)&=&f(u+2n\pi ,v)=f(u, v+2n\pi)
\end{eqnarray*}
holds for $n=0, \pm 1,\pm 2,\ldots.$ $C_{2\pi}(\mathbb{R}^2)$ is the set of all $2\pi$-periodic and real valued continuous functions on   $\mathbb{R}^2$ with the norm
\begin{eqnarray*}
\Vert f\Vert=\sup_{(u,v)\in\mathbb{R}^2}|f(u,v)| \qquad (f\in C(\mathbb{R}^2)).
\end{eqnarray*}
For fuzzy valued function $f:K \subset\mathbb{R}^2\to \mathbb{R}_{\mathcal{F}}$, $\alpha$-level set of $f$ is
\begin{eqnarray*}
\left[f(u,v)\right]_{\alpha}=\left[f^-_{\alpha}(u,v),f^+_{\alpha}(u,v)\right]
\end{eqnarray*}
for every $(u,v)\in K$ and $\alpha\in[0,1]$. $C_{2\pi}^{(\mathcal{F})}(\mathbb{R}^2)$ is the set of all $2\pi$-periodic and continuous fuzzy valued functions on $\mathbb{R}^2$ and the continuity is with respect to metric $D$. $C_{2\pi}^{(\mathcal{F})}(\mathbb{R}^2)$ is equipped with the metric
\begin{eqnarray*}
D^{*}(f,g)=\sup_{(u,v)\in \mathbb{R}^2}D(f(u,v),g(u,v))=\sup_{(u,v)\in \mathbb{R}^2}\sup_{\alpha\in[0,1]}\max\{|f^-_{\alpha}(u,v)-g^-_{\alpha}(u,v)|,|f^+_{\alpha}(u,v)-g^+_{\alpha}(u,v)|\}.
\end{eqnarray*}
For $f\in  C_{2\pi}^{(\mathcal{F})}(\mathbb{R}^2)$, the modulus of continuity is given by
\begin{eqnarray*}
\omega^{\mathcal{F}}(f; \delta):=\sup\{D\left(f(u,v), f(x,y)\right): (u,v),(x,y)\in\mathbb{R}^2,  \sqrt{(u-x)^2+(v-y)^2}\leq \delta\}
\end{eqnarray*}
for any $\delta>0$.

We now prove following lemma.
\begin{lemma}\label{uniform1}
If $f\in  C_{2\pi}^{(\mathcal{F})}(\mathbb{R}^2)$, then for all $\delta\in[0,\pi]$ we have
\begin{eqnarray*}
\omega^{\mathcal{F}}(f\big|_{[0,2\pi]\times[0,2\pi]}; \delta)\leq\omega^{\mathcal{F}}(f; \delta)\leq 3\omega^{\mathcal{F}}(f\big|_{[0,2\pi]\times[0,2\pi]}; \delta).
\end{eqnarray*}
\end{lemma}
\begin{proof}
For convenience denote $(x,y),(u,v)\in\mathbb{R}^2$ with $z=(x,y),w=(u,v)$ and $d(z,w)=\sqrt{(x-u)^2+(y-v)^2}$. The left hand side inequality is obvious. For the right side, let us part $\mathbb{R}^2$ as $I_{mn}=[2m\pi, 2(m+1)\pi]\times[2n\pi, 2(n+1)\pi]$ where $m,n\in\mathbb{Z}$. Then for any $z,w\in\mathbb{R}^2$ with $d(z,w)\leq \delta$, there exist following possibilities:
\begin{enumerate}
\item[(1)] $\exists m,n\in\mathbb{Z}$ such that $z,w\in I_{mn}$
\item[(2)] $\exists m,n\in\mathbb{Z}$ such that $z\in I_{mn}, w\in I_{m+1,n}$ or $z\in I_{m+1,n}, w\in I_{mn}$
\item[(3)] $\exists m,n\in\mathbb{Z}$ such that $z\in I_{mn}, w\in I_{m,n+1}$ or $z\in I_{m,n+1}, w\in I_{mn}$
\item[(4)] $\exists m,n\in\mathbb{Z}$ such that $z\in I_{mn}, w\in I_{m+1,n+1}$ or $z\in I_{m+1,n+1}, w\in I_{mn}$
\item[(5)] $\exists m,n\in\mathbb{Z}$ such that $z\in I_{mn}, w\in I_{m-1,n+1}$ or $z\in I_{m-1,n+1}, w\in I_{mn}$
\end{enumerate}
We do the proofs for only the first parts of the cases since the proofs for the second parts are done similarly by changing the roles of $z$ and $w$. Before to continue with the proofs, let $z'=(x-2m\pi,y-2n\pi), w'= (u-2m\pi,v-2n\pi)$ and denote the set of all points on the line segment connecting the points $z'$ and $w'$ with $\ell$. We note that $d(z',w')=d(z,w)\leq \delta$.

{\it Case (1).} $z,w\in I_{mn}$ and so we have $z',w'\in I_{00}$. Then from periodicity of $f$ we get
\begin{eqnarray*}
D(f(z),f(w))=D(f(z'),f(w'))\leq \omega^{\mathcal{F}}(f\big|_{I_{00}}; \delta)\leq3\omega^{\mathcal{F}}(f\big|_{I_{00}}; \delta).
\end{eqnarray*}

{\it Case (2).} $z\in I_{mn}, w\in I_{m+1,n}$ and so we have $z'\in I_{00}, w'\in I_{10}$. From periodicity of $f$ we get
\begin{eqnarray*}
D(f(z),f(w))&=&D(f(z'),f(w'))
\\&\leq&
D(f(z'),f(\ell'))+D(f(\ell'),f(w'))
\\&\leq&
\omega^{\mathcal{F}}(f\big|_{I_{00}}; \delta)+\omega^{\mathcal{F}}(f\big|_{I_{10}}; \delta)
\\ &=&
2\omega^{\mathcal{F}}(f\big|_{I_{00}};\delta)\leq3\omega^{\mathcal{F}}(f\big|_{I_{00}};\delta)
\end{eqnarray*}
where $\ell'=\ell\cap I_{00}\cap I_{10}$.

{\it Cases (3).} A procedure analogous to that in case (2) can be applied by choosing $\ell'=\ell\cap I_{00}\cap I_{01}$.

{\it Cases (4).} $z\in I_{mn}, w\in I_{m+1,n+1}$ and so we have $z'\in I_{00}, w'\in I_{11}$. If $(2\pi,2\pi)\in\ell$ then
\begin{eqnarray*}
D(f(z),f(w))&=&D(f(z'),f(w'))
\\&\leq&
D(f(z'),f(2\pi,2\pi))+D(f(2\pi,2\pi),f(w'))
\\&\leq&
\omega^{\mathcal{F}}(f\big|_{I_{00}}; \delta)+\omega^{\mathcal{F}}(f\big|_{I_{11}}; \delta)
\\&=&
2\omega^{\mathcal{F}}(f\big|_{I_{00}};\delta)\leq3\omega^{\mathcal{F}}(f\big|_{I_{00}};\delta).
\end{eqnarray*}
If $(2\pi,2\pi)\not\in\ell$ and $\ell\cap I_{10}\neq\varnothing$ then we choose $\ell'=\ell\cap I_{00}\cap I_{10}$ and $\ell''=\ell\cap I_{10}\cap I_{11}$. Then we get
\begin{eqnarray*}
D(f(z),f(w))&=&D(f(z'),f(w'))
\\&\leq&
D(f(z'),f(\ell'))+D(f(\ell'),f(\ell''))+D(f(\ell''),f(w'))
\\&\leq&
\omega^{\mathcal{F}}(f\big|_{I_{00}}; \delta)+\omega^{\mathcal{F}}(f\big|_{I_{10}}; \delta)+\omega^{\mathcal{F}}(f\big|_{I_{11}}; \delta)
\\&=&
3\omega^{\mathcal{F}}(f\big|_{I_{00}};\delta).
\end{eqnarray*}
If $(2\pi,2\pi)\not\in\ell$ and $\ell\cap I_{10}=\varnothing$, then we choose $\ell'=\ell\cap I_{00}\cap I_{01}$, $\ell''=\ell\cap I_{01}\cap I_{11}$ and get
\begin{eqnarray*}
D(f(z),f(w))&=&D(f(z'),f(w'))
\\&\leq&
D(f(z'),f(\ell'))+D(f(\ell'),f(\ell''))+D(f(\ell''),f(w'))
\\&\leq&
\omega^{\mathcal{F}}(f\big|_{I_{00}}; \delta)+\omega^{\mathcal{F}}(f\big|_{I_{01}}; \delta)+\omega^{\mathcal{F}}(f\big|_{I_{11}}; \delta)
\\&=&
3\omega^{\mathcal{F}}(f\big|_{I_{00}};\delta).
\end{eqnarray*}
{\it Cases (5).} This case is similar to the case (4), hence omitted.

In all cases, taking supremum over $\{z,w\in\mathbb{R}^2: d(z,w)\leq \delta\}$ we get the claim.
\end{proof}
By Lemma \ref{uniform1}, we prove uniform continuity of $f\in  C_{2\pi}^{(\mathcal{F})}(\mathbb{R}^2)$.
\begin{theorem}\label{uniform2}
If $f\in  C_{2\pi}^{(\mathcal{F})}(\mathbb{R}^2)$, then $f$ is bounded and uniformly continuous.
\end{theorem}
\begin{proof}
Let $f\in  C_{2\pi}^{(\mathcal{F})}(\mathbb{R}^2)$. By Lemma 10.8 in \cite[p. 131]{anastassiou} we have
\begin{eqnarray*}
D(f(u,v),\bar{0})\leq M, \quad \forall (u,v)\in [0,2\pi]\times[0,2\pi],\quad M>0.
\end{eqnarray*}
For any $(x,y)\not\in[0,2\pi]\times[0,2\pi]$  there exists $(u,v)\in[0,2\pi]\times[0,2\pi]$ such that $(x,y)=(u+2m\pi, v+2n\pi)$ where $m,n\in\mathbb{Z}$. Hence we have
\begin{eqnarray*}
D(f(x,y),\bar{0})=D(f(u,v),\bar{0})\leq M, \quad \forall (x,y)\in \mathbb{R}^2\backslash[0,2\pi]\times[0,2\pi]
\end{eqnarray*}
which proves that fuzzy valued function $f$ is bounded on $\mathbb{R}^2$.

$f$ is uniformly continuous on $[0,2\pi]\times[0,2\pi]$ by Heine-Cantor theorem. Then, by $(v)$ of Proposition 10.7 in \cite[p. 129]{anastassiou}, we have
$\lim_{\delta\to 0}\omega^{\mathcal{F}}(f\big|_{[0,2\pi]\times[0,2\pi]};\delta)=0$. Thus by Lemma \ref{uniform1} we get
\begin{eqnarray*}
\lim_{\delta\to 0}\omega^{\mathcal{F}}(f;\delta)=0,
\end{eqnarray*}
and this implies the uniform continuity of $f$ on $\mathbb{R}^2$ in view of $(v)$ of Proposition 10.7 in \cite[p. 129]{anastassiou}.
\end{proof}
In view of Theorem \ref{uniform2} above and Proposition 10.7($i$)-Proposition 10.3 in \cite[pp. 128--129]{anastassiou} we have the following theorem.
\begin{theorem}\cite{anastassiou}\label{thmmodulus1}
If  $f\in  C_{2\pi}^{(\mathcal{F})}(\mathbb{R}^2)$, then for any $\delta>0$
\begin{eqnarray*}
\omega^{\mathcal{F}}(f; \delta)=\sup_{\alpha\in[0,1]}\max\left\{\omega(f^-_{\alpha}; \delta), \omega(f^+_{\alpha}; \delta)\right\}.
\end{eqnarray*}
\end{theorem}

Let $T:C_{2\pi}^{(\mathcal{F})}(\mathbb{R}^2)\rightarrow C_{2\pi}^{(\mathcal{F})}(\mathbb{R}^2)$. $T$ is called {\it fuzzy linear} operator if, for every $k_{1},k_{2}\in \mathbb{R}$, $f_{1},f_{2}\in C_{2\pi}^{(\mathcal{F})}(\mathbb{R}^2)$ and $(u,v)\in \mathbb{R}^2$,
\begin{equation*}
T(k_{1} f_{1}+k_{2} f_{2};u,v) = k_{1}  T(f_{1};u,v) + k_{2}  T(f_{2};u,v)
\end{equation*}
is satisfied. $T$ is said to be {\it fuzzy positive linear} operator if it is {\it fuzzy linear} and $f(u,v) \preceq g(u,v)\Rightarrow T(f;u,v) \preceq T(g;u,v)$ holds for any $f,g \in C_{2\pi}^{(\mathcal{F})}(\mathbb{R}^2)$ and $(u,v) \in \mathbb{R}^2$.
\begin{theorem}\cite{double}
If $f:[a,b]\times[c,d]\to \mathbb{R}_{\mathcal{F}}$ is continuous, then
\begin{eqnarray*}
\int_{c}^{d}\int_{a}^{b}f(u,v)dudv
\end{eqnarray*}
exists and
\begin{eqnarray*}
\left[\int_{c}^{d}\int_{a}^{b}f(u,v)dudv\right]_{\alpha}=\left[\int_{c}^{d}\int_{a}^{b}f^{-}_{\alpha}(u,v)dudv, \int_{c}^{d}\int_{a}^{b}f^{+}_{\alpha}(u,v)dudv\right]
\end{eqnarray*}
where the integrals are in the fuzzy Riemann sense.
\end{theorem}
\section{Main Results}
Now we state our main results. Note that limit and convergence of double sequences are meant in the Pringsheim sense\cite{pringsheim}.
\begin{theorem}\label{thmbasic}
 Let $\{T_{mn}\}$ be a sequence of {\it fuzzy positive linear} operators from $C_{2\pi}^{(\mathcal{F})}(\mathbb{R}^2)$ into itself and there be a corresponding sequence $\{\tilde{T}_{mn}\}$ of positive linear operators from $C_{2\pi}(\mathbb{R}^2)$ into itself satisfying
\begin{eqnarray}\label{tilde}
\left\{T_{mn}(f;u,v)\right\}_\alpha^\mp=\widetilde{T}_{mn}(f_\alpha^\mp;u,v)
\end{eqnarray}
for all $(u,v) \in \mathbb{R}^2, \alpha \in [0, 1], $ and $f \in C_{2\pi}^{(\mathcal{F})}(\mathbb{R}^2)$. If
\begin{equation}\label{hypothesis}
\lim_{m,n\rightarrow\infty}\Vert\widetilde{T}_{mn}(f_{i})-f_{i}\Vert=0,
\end{equation}
holds for $f_0(u,v)=1, f_1(u,v)=\sin u, f_2(u,v)=\sin v,f_3(u,v)=\cos u,f_4(u,v)=\cos v$, then for all $f \in C_{2\pi}^{(\mathcal{F})}(\mathbb{R}^2)$
\begin{equation*}
\lim_{m,n\rightarrow\infty} D^{*}\left(T_{mn}(f),f\right)=0.
\end{equation*}
\end{theorem}
\begin{proof}
Let $f\in C_{2\pi}^{(\mathcal{F})}(\mathbb{R}^2)$ and condition (\ref{hypothesis}) hold. Assume that $I$ and $J$ are closed subintervals of length $2\pi$ of $\mathbb{R}$. Fix $(u,v)\in I\times J$. Then, for given $\varepsilon>0$ there exists $\delta>0$ so that
\begin{eqnarray*}
|f_{\alpha}^{\mp}(x,y)-f_{\alpha}^{\mp}(u,v)|<\varepsilon+\frac{2B_{\alpha}^{\mp}}{\sin^2\left(\frac{\delta}{2}\right)}\left\{\sin^2\left(\frac{x-u}{2}\right)+\sin^2\left(\frac{y-v}{2}\right)\right\}
\end{eqnarray*}
where $B_{\alpha}^{\mp}:=\Vert f_{\alpha}^{\mp}\Vert$. Then, we have
\begin{eqnarray*}
&&\left|\widetilde{T}_{mn}(f^{\mp}_{\alpha};u,v)-f^{\mp}_{\alpha}(u,v)\right|
\\&\leq&
\widetilde{T}_{mn}(|f^{\mp}_{\alpha}(x,y)-f^{\mp}_{\alpha}(u,v)|;u,v) +B_{\alpha}^{\mp}\left|\widetilde{T}_{mn}(f_0;u,v)-f_0(u,v)\right|
\\&\leq&
\widetilde{T}_{mn}\left(\varepsilon+\frac{2B_{\alpha}^{\mp}}{\sin^2\left(\frac{\delta}{2}\right)}\left\{\sin^2\left(\frac{x-u}{2}\right)+\sin^2\left(\frac{y-v}{2}\right)\right\};u,v\right)+B_{\alpha}^{\mp}\left|\widetilde{T}_{mn}(f_0;u,v)-f_0(u,v)\right|
\\&\leq&
\varepsilon+\left(\varepsilon+B_{\alpha}^{\mp}\right)\left|\widetilde{T}_{mn}(f_0;u,v)-f_0(u,v)\right|+\frac{2B_{\alpha}^{\mp}}{\sin^2\left(\frac{\delta}{2}\right)}\widetilde{T}_{mn}\left(\sin^2\left(\frac{x-u}{2}\right)+\sin^2\left(\frac{y-v}{2}\right);u,v\right)
\\&\leq&
\varepsilon+(\varepsilon+B_{\alpha}^{\mp})\left|\widetilde{T}_{mn}(f_0;u,v)-f_0(u,v)\right|+\frac{B_{\alpha}^{\mp}}{\sin^2\left(\frac{\delta}{2}\right)}\left\{2\left| \widetilde{T}_{mn}(f_0;u,v)-f_0(u,v)\right|\right.
\\&&\qquad\qquad\qquad\qquad\qquad\qquad\qquad\qquad\qquad\qquad\qquad\quad
+|\sin u|\left|\widetilde{T}_{mn}(f_1;u,v)-f_1(u,v)\right|
\\&& \qquad\qquad\qquad\qquad\qquad\qquad\qquad\qquad\qquad\qquad\qquad\quad
 +|\sin v|\left|\widetilde{T}_{mn}(f_2;u,v)-f_2(u,v)\right|
 \\&&\qquad\qquad\qquad\qquad\qquad\qquad\qquad\qquad\qquad\qquad\qquad\quad
+|\cos u|\left|\widetilde{T}_{mn}(f_3;u,v)-f_3(u,v)\right|
\\&&\qquad\qquad\qquad\qquad\qquad\qquad\qquad\qquad\qquad\qquad\qquad\quad
 \left. +|\cos v|\left|\widetilde{T}_{mn}(f_4;u,v)-f_4(u,v)\right|\right\}
\\&\leq&
 \varepsilon+H_{\alpha}^{\mp}(\varepsilon)\left\{\sum_{i=0}^4 \left|\widetilde{T}_{mn}(f_i;u,v)-f_i(u,v)\right|\right\}
\end{eqnarray*}%
where $H_{\alpha}^{\mp}(\varepsilon):=\varepsilon+ B_{\alpha}^{\mp}+ 2B_{\alpha}^{\mp}\big/\sin^2(\delta/2)$. Considering condition \eqref{tilde} if we take $\sup_{\alpha\in[0,1]}\max$ we get
\begin{eqnarray*}
D\left(T_{mn}(f;u,v),f\right)\leq \varepsilon+H(\varepsilon)\left\{\sum_{i=0}^4 \left|\widetilde{T}_{mn}(f_i;u,v)-f_i(u,v)\right|\right\}
\end{eqnarray*}
where $H(\varepsilon):=\sup\limits_{\alpha\in[0,1]}\max\left\{H_{\alpha}^{-}(\varepsilon),H_{\alpha}^{+}(\varepsilon)\right\}$. Finally taking supremum over $(u,v)$ we get
\begin{eqnarray*}
D^*\left(T_{mn}(f),f\right)\leq \varepsilon+H(\varepsilon)\left\{\sum_{i=0}^4 \left\Vert\widetilde{T}_{mn}(f_i)-f_i\right\Vert\right\}.
\end{eqnarray*}
This implies that $D^*\left(T_{mn}(f),f\right)\to 0$ as $m,n\to \infty$ by condition \eqref{hypothesis} of the theorem. The proof is completed.
\end{proof}
\begin{theorem}\label{thmmodulus2}
Let $(T_{mn})$ be a sequence of {\it fuzzy positive linear} operators from $C_{2\pi}^{(\mathcal{F})}(\mathbb{R}^2)$ to  $C_{2\pi}^{(\mathcal{F})}(\mathbb{R}^2)$ so that condition \eqref{tilde} holds. If
\begin{itemize}
\item[(i)] $\lim\limits_{m,n\rightarrow\infty}\Vert\widetilde{T}_{mn}(f_{0})-f_{0}\Vert=0$, \hspace{1cm} $\left(f_{0}(u,v)=1\right)$
\item[(ii)] $\lim\limits_{m,n\rightarrow\infty}\omega^{\mathcal{F}}(f; \gamma_{mn})=0$ where $\gamma_{mn}=\sqrt{\left\Vert\widetilde{T}_{mn}(\varphi)\right\Vert}$ with $\varphi(u,v)=\sin^2\left(\frac{x-u}{2}\right)+\sin^2\left(\frac{y-v}{2}\right)$,
\end{itemize}
then for all $f\in C_{2\pi}^{(\mathcal{F})}(\mathbb{R}^2)$
\begin{eqnarray*}
\lim_{m,n\rightarrow\infty}D^*\left(T_{mn}(f), f\right)=0.
\end{eqnarray*}
\end{theorem}
\begin{proof}
Assume $f\in C_{2\pi}^{(\mathcal{F})}(\mathbb{R}^2)$  and $(u,v)\in[-\pi, \pi]\times[-\pi, \pi]$. Then, $f^{\mp}_{\alpha}\in C_{2\pi}(\mathbb{R}^2)$ and from the proof of Theorem 9 in \cite{demirci} we get
\begin{eqnarray*}
|f_{\alpha}^{\mp}(x,y)-f_{\alpha}^{\mp}(u,v)|\leq \left(1+\pi^2\frac{\sin^2\left(\frac{x-u}{2}\right)+\sin^2\left(\frac{y-v}{2}\right)}{\delta^2}\right)\omega(f_{\alpha}^{\mp}; \delta)
\end{eqnarray*}
and
\begin{eqnarray*}
&&\left|\widetilde{T}_{mn}(f^{\mp}_{\alpha};u,v)-f^{\mp}_{\alpha}(u,v)\right|\leq \widetilde{T}_{mn}\left(|f^{\mp}_{\alpha}(x,y)-f^{\mp}_{\alpha}(u,v)|;u,v\right) +B_{\alpha}^{\mp}\left|\widetilde{T}_{mn}(f_0;u,v)-f_0(u,v)\right|
\\&\leq&
\omega(f_{\alpha}^{\mp};\delta)\widetilde{T}_{mn}\left(1+(\pi/\delta)^2\left\{\sin^2\left(\frac{x-u}{2}\right)\!+\!\sin^2\left(\frac{y-v}{2}\right)\right\};u,v\right) +B_{\alpha}^{\mp}\left|\widetilde{T}_{mn}(f_0;u,v)-f_0(u,v)\right|
\\&\leq&
\left|\widetilde{T}_{mn}(f_0;u,v)-f_0(u,v)\right|\omega(f_{\alpha}^{\mp};\delta) +\omega(f_{\alpha}^{\mp};\delta) +\frac{\pi^2}{\delta^2}\omega(f_{\alpha}^{\mp};\delta)\widetilde{T}_{mn}\left(\varphi;u,v\right)+B_{\alpha}^{\mp}\left|\widetilde{T}_{mn}(f_0;u,v)-f_0(u,v)\right|
\\&\leq&
\left|\widetilde{T}_{mn}(f_0;u,v)-f_0(u,v)\right|\omega(f_{\alpha}^{\mp};\gamma_{mn})+(1+\pi^2)\omega(f_{\alpha}^{\mp};\gamma_{mn})+B_{\alpha}^{\mp}\left|\widetilde{T}_{mn}(f_0;u,v)-f_0(u,v)\right|
\end{eqnarray*}
by putting $\delta:=\gamma_{mn}$ in the last part. Then, taking $\sup_{\alpha\in[0,1]}\max$ of both sides we get
\begin{eqnarray*}
D\left(T_{mn}(f;u,v),f\right)\leq H\left\{\left|\widetilde{T}_{mn}(f_0;u,v)-f_0(u,v)\right|\omega^{\mathcal{F}}(f;\gamma_{mn})+\omega^{\mathcal{F}}(f;\gamma_{mn})+\left|\widetilde{T}_{mn}(f_0;u,v)-f_0(u,v)\right|\right\}
\end{eqnarray*}
where $H=\max\{1+\pi^2, \sup_{\alpha\in[0,1]}\max\{B_{\alpha}^{-},B_{\alpha}^{+}\}\}$ in view of Thoerem \ref{thmmodulus1}. Hence, we get
\begin{eqnarray*}
D^*\left(T_{mn}(f),f\right)\leq H\left\{\Vert\widetilde{T}_{mn}(f_0)-f_0\Vert\omega^{\mathcal{F}}(f;\gamma_{mn})+\omega^{\mathcal{F}}(f;\gamma_{mn})+\Vert\widetilde{T}_{mn}(f_0)-f_0\Vert\right\}
\end{eqnarray*}
Taking limits of both sides as $m,n\to\infty$ and considering conditions $(i)-(ii)$ of Theorem \ref{thmmodulus2}, we complete the proof.
\end{proof}
\section{Approximation via double level Fourier series of fuzzy valued functions}
Following \cite[Definition 4.1]{levelfourier}, we introduce double level Fourier series of fuzzy valued functions of two variables.
\begin{definition}
Let $f\in C_{2\pi}^{(\mathcal{F})}(\mathbb{R}^2)$. Double level Fourier series of $f$ is given by the pair of series
\begin{eqnarray}\label{levelfourier}
\begin{split}
\frac{a^{\mp}_{00}(\alpha)}{4}&+\frac{1}{2}\sum_{j=1}^{\infty}(a^{\mp}_{j0}(\alpha)\cos ju+c^{\mp}_{j0}(\alpha)\sin ju) +\frac{1}{2}\sum_{k=1}^{\infty}(a^{\mp}_{0k}(\alpha)\cos kv+b^{\mp}_{0k}(\alpha)\sin kv)
\\&+\mathlarger{\sum_{j=1}^{\infty}\sum_{k=1}^{\infty}}\left\{
   \begin{array}{ll}
     a^{\mp}_{jk}(\alpha)\cos ju\cos kv+b^{\mp}_{jk}(\alpha)\cos ju\sin kv\\
     +c^{\mp}_{jk}(\alpha)\sin ju\cos kv+d^{\mp}_{jk}(\alpha)\sin ju\sin kv
   \end{array}
   \right\}
   \end{split}
\end{eqnarray}
where
\begin{eqnarray*}
&&a^{\mp}_{jk}(\alpha)=\frac{1}{\pi^2}\int_{-\pi}^{\pi}\int_{-\pi}^{\pi}f^{\mp}_{\alpha}(u,v)\cos ju\cos kv dudv, \ \ b^{\mp}_{jk}(\alpha)=\frac{1}{\pi^2}\int_{-\pi}^{\pi}\int_{-\pi}^{\pi}f^{\mp}_{\alpha}(u,v)\cos ju\sin kv dudv
\\
&&c^{\mp}_{jk}(\alpha)=\frac{1}{\pi^2}\int_{-\pi}^{\pi}\int_{-\pi}^{\pi}f^{\mp}_{\alpha}(u,v)\sin ju\cos kv dudv, \ \ d^{\mp}_{jk}(\alpha)=\frac{1}{\pi^2}\int_{-\pi}^{\pi}\int_{-\pi}^{\pi}f^{\mp}_{\alpha}(u,v)\sin ju\sin kv dudv.
\end{eqnarray*}
"-" and "+" signed series in \eqref{levelfourier} are said to be left and right double level Fourier series of $f$, respectively.
\end{definition}
\noindent Levelwise approach to Fourier series of fuzzy valued functions provides researchers convenience in approximation of fuzzy valued functions through level sets. While being more advantageous in calculations via level sets when compared to fuzzy Fourier series\cite{fuzzyfourier}, level Fourier series(as pair of functions with variable $\alpha$) may not represent a fuzzy number. Besides, like the Fourier series in both classical and fuzzy settings, level Fourier series of fuzzy valued functions may not converge to the fuzzy valued function in hand. In one dimensional case\cite{levelfourier}, summation methods are used to achieve the fuzzification of level Fourier series and fuzzy trigonometric Korovkin approximation theorems are utilized to recover the convergence. Analogous way may be utilized also in two dimensional case to deal with fuzzification and convergence of double level Fourier series of fuzzy valued functions. In this way, now we define the double fuzzy Fej\'{e}r operator and double fuzzy Abel-Poisson convolution operator by means of taking Ces\`{a}ro and Abel means of the double level Fourier series, respectively. Here we remind the reader that Ces\`{a}ro and Abel means of a double series $\sum_{j,k=0}^{\infty}u_{jk}$ is defined by
\begin{eqnarray*}
\sigma_{mn}=\frac{1}{(m+1)(n+1)}\sum_{j=0}^{m}\sum_{k=0}^{n}t_{jk}\qquad \tau_{rs}=\sum_{j=0}^{\infty}\sum_{k=0}^{\infty}u_{jk}r^js^k
\end{eqnarray*}
where $\{t_{mn}\}$ is the sequence of partial sums of $\sum_{j,k=0}^{\infty}u_{jk}$ and $r,s\in(0,1)$. For more information concerning summation methods with applications see \cite{boos,mursaleensum,bor1,bor2}.

Let $f:\mathbb{R}^2\to \mathbb{R}_{\mathcal{F}}$. Double fuzzy Fej\'{e}r operator is defined by
\begin{eqnarray*}
F_{mn}(f;u,v)=\frac{1}{4\pi^2}\int_{-\pi}^{\pi}\int_{-\pi}^{\pi}f(u-x,v-y)\phi_m(x)\phi_n(y)dxdy
\end{eqnarray*}
where
\begin{eqnarray*}
\phi_m(u)=
\begin{cases}
\frac{\sin^2[(m+1)u/2]}{(m+1)\sin^2[u/2]} \quad &\text{if $u$ is not a multiple of $2\pi$}\\[1mm]
m+1  &\text{if $u$ is a multiple of $2\pi$.}
\end{cases}
\end{eqnarray*}
Now we apply Theorem \ref{thmbasic} to the double fuzzy Fej\'{e}r operator. Since Fej\'{e}r kernel $\{\phi_m\}$ is a positive kernel, $\{F_{mn}\}$ defines a sequence of {\it fuzzy positive linear} operators and there corresponds a sequence of positive linear operators $\{\widetilde{F}_{mn}\}$ such that
\begin{eqnarray*}
\left\{F_{mn}(f;u,v)\right\}_\alpha^\mp=\widetilde{F}_{mn}(f_\alpha^\mp;u,v)=\frac{1}{4\pi^2}\int_{-\pi}^{\pi}\int_{-\pi}^{\pi}f_\alpha^\mp(u-x,v-y)\phi_m(x)\phi_n(y)dxdy
\end{eqnarray*}
which satisfies the property \eqref{tilde}. Furthermore, since
\begin{eqnarray*}
&&\widetilde{F}_{mn}(f_0;u,v)=f_0\\
&&\widetilde{F}_{mn}(f_1;u,v)=\frac{m}{m+1}f_1(u,v) \qquad \widetilde{F}_{mn}(f_2;u,v)=\frac{n}{n+1}f_2(u,v)\\
&&\widetilde{F}_{mn}(f_3;u,v)=\frac{m}{m+1}f_3(u,v) \qquad \widetilde{F}_{mn}(f_4;u,v)=\frac{n}{n+1}f_4(u,v),
\end{eqnarray*}
we have
\begin{eqnarray*}
\Vert\widetilde{F}_{mn}(f_{i})-f_{i}\Vert=
\begin{cases}
0, \quad &if \quad i=0\\
\frac{1}{m+1},  &if \quad i=1,3\\
\frac{1}{n+1} &if\quad i=2,4
\end{cases}
\end{eqnarray*}
and this implies $\lim_{m,n}\Vert\widetilde{F}_{mn}(f_{i})-f_{i}\Vert=0$. Hence, \eqref{tilde} and \eqref{hypothesis} of Theorem \ref{thmbasic} are fulfilled and we conclude $F_{mn}(f)\to f$ uniformly. We note that convergence of sequence $\{F_{mn}\}$ may also be verified  by Theorem \ref{thmmodulus2}. We have $\lim_{m,n}\Vert\widetilde{F}_{mn}(f_{0})-f_{0}\Vert=0$, and since
\begin{eqnarray*}
\lim\limits_{m,n\rightarrow\infty}\gamma_{mn}=\lim\limits_{m,n\rightarrow\infty}\sqrt{\left\Vert\widetilde{F}_{mn}\left(\sin^2\left(\frac{x-u}{2}\right)+\sin^2\left(\frac{y-v}{2}\right);u,v\right)\right\Vert}=
\lim\limits_{m,n\rightarrow\infty}\sqrt{\frac{1}{2(m+1)}+\frac{1}{2(n+1)}}=0,
\end{eqnarray*}
we have $\lim_{m,n}\omega^{\mathcal{F}}(f; \gamma_{mn})=0$ by uniform continuity of $f$ on $\mathbb{R}^2$(see Theorem \ref{uniform2}). Hence $(i)-(ii)$ of Theorem \ref{thmmodulus2} are satisfied and this implies $F_{mn}(f)\to f$ uniformly.

Now we define double fuzzy Abel-Poisson convolution operator via Abel means of double level Fourier series. Let $f:\mathbb{R}^2\to \mathbb{R}_{\mathcal{F}}$. Double fuzzy Abel-Poisson operator is defined by
\begin{eqnarray*}
P_{rs}(f;u,v)=\frac{1}{4\pi^2}\int_{-\pi}^{\pi}\int_{-\pi}^{\pi}f(u-x,v-y)\psi_r(x)\psi_s(y)dxdy \qquad (0\leq r <1,\ 0\leq s <1)
\end{eqnarray*}
where $\psi_r(u)=\frac{1-r^2}{1-2r\cos u+r^2}\cdot$ We claim that $\lim_{r,s\to1^-}P_{rs}(f)=f$. We use the sequential criterion for the limit in metric spaces to investigate the limit $\lim_{r,s\to1^-}P_{rs}(f)$ and hence we consider the sequence of operators
\begin{eqnarray*}
T_{mn}(f;u,v)=\frac{1}{4\pi^2}\int_{-\pi}^{\pi}\int_{-\pi}^{\pi}f(u-x,v-y)\psi_{r_m}(x)\psi_{s_n}(y)dxdy
\end{eqnarray*}
where $(r_m)$ and $(s_n)$ are arbitrary sequences on $[0,1)$ such that $\lim_{m} r_m=1$ and $\lim_{n} s_n=1$. Now we again apply Theorem \ref{thmbasic}. As the case in Fej\'{e}r kernels, Abel-Poisson kernel $\{\psi_{r_m}\}$ is also positive and hence $\{T_{mn}\}$ defines a sequence of {\it fuzzy positive linear} operators and there corresponds a sequence of positive linear operators $\{\widetilde{T}_{mn}\}$ such that
\begin{eqnarray*}
\left\{T_{mn}(f;u,v)\right\}_\alpha^\mp=\widetilde{T}_{mn}(f_\alpha^\mp;u,v)=\frac{1}{4\pi^2}\int_{-\pi}^{\pi}\int_{-\pi}^{\pi}f_\alpha^\mp(u-x,v-y)\psi_{r_m}(x)\psi_{s_n}(y)dxdy.
\end{eqnarray*}
Furthermore, since
\begin{eqnarray*}
&&\widetilde{T}_{mn}(f_0;u,v)=f_0\\
&&\widetilde{T}_{mn}(f_1;u,v)=r_mf_1(u,v) \qquad \widetilde{T}_{mn}(f_2;u,v)=s_nf_2(u,v)\\
&&\widetilde{T}_{mn}(f_3;u,v)=r_mf_3(u,v) \qquad \widetilde{T}_{mn}(f_4;u,v)=s_nf_4(u,v),
\end{eqnarray*}
we have
\begin{eqnarray*}
\Vert\widetilde{T}_{mn}(f_{i})-f_{i}\Vert=
\begin{cases}
0, \quad &if \quad i=0\\
1-r_m,  &if \quad i=1,3\\
1-s_n &if\quad i=2,4
\end{cases}
\end{eqnarray*}
and this implies $\lim_{m,n}\Vert\widetilde{T}_{mn}(f_{i})-f_{i}\Vert=0$. Hence, \eqref{tilde} and \eqref{hypothesis} of Theorem \ref{thmbasic} are fulfilled and we obtain that $T_{mn}(f)\to f$ uniformly. Since $(r_m)$ and $(s_n)$ are arbitrary, we conclude $\lim_{r,s\to1^-}P_{rs}(f)=f$. Here we note that the approximation $T_{mn}(f)\to f$ may also be checked by Theorem \ref{thmmodulus2}. We have $\lim_{m,n}\Vert\widetilde{T}_{mn}(f_{0})-f_{0}\Vert=0$, and by the fact that
\begin{eqnarray*}
\lim\limits_{m,n\rightarrow\infty}\gamma_{mn}=\lim\limits_{m,n\rightarrow\infty}\sqrt{\left\Vert\widetilde{T}_{mn}\left(\sin^2\left(\frac{x-u}{2}\right)+\sin^2\left(\frac{y-v}{2}\right);u,v\right)\right\Vert}=
\lim\limits_{m,n\rightarrow\infty}\sqrt{\frac{1-r_m}{2}+\frac{1-s_n}{2}}=0,
\end{eqnarray*}
we have $\lim_{m,n}\omega^{\mathcal{F}}(f; \gamma_{mn})=0$ by virtue of uniform continuity of $f$ on $\mathbb{R}^2$(see Theorem \ref{uniform2}). Hence $(i)-(ii)$ of Theorem \ref{thmmodulus2} are satisfied and this implies $T_{mn}(f)\to f$ uniformly. Again by the sequential criterion for limit we have $\lim_{r,s\to1^-}P_{rs}(f)=f$.

\end{document}